\input ./style/arxiv-vmsta.cfg
\documentclass[numbers,compress,v1.0.1]{vmsta}
\usepackage{vtexbibtags}

\volume{6}
\issue{3}
\pubyear{2019}
\firstpage{311}
\lastpage{331}
\aid{VMSTA137}
\doi{10.15559/19-VMSTA137}
\articletype{research-article}

\startlocaldefs

\allowdisplaybreaks
\numberwithin{equation}{section}
\theoremstyle{plain}
\newtheorem{thm}{Theorem}[section]
\newtheorem{prop}{Proposition}[section]
\newtheorem{lem}{Lemma}[section]
\theoremstyle{definition}
\newtheorem{rem}{Remark}[section]
\newtheorem{exper}{Experiment}[section]
\ifdefined\HCode 
\def\index#1{}
\else 
\fi
\endlocaldefs

\begin{document}

\begin{frontmatter}
\pretitle{Research Article}

\title{Taylor's power law for the $N$-stars network evolution model}

\author{\inits{I.}\fnms{Istv\'{a}n}~\snm{Fazekas}\thanksref{cor1}\ead[label=e1]{fazekas.istvan@inf.unideb.hu}}
\author{\inits{Cs.}\fnms{Csaba}~\snm{Nosz\'{a}ly}\ead[label=e2]{noszaly.csaba@inf.unideb.hu}}
\author{\inits{N.}\fnms{No\'{e}mi}~\snm{Uzonyi}\ead[label=e2]{unomi95@gmail.com}}
\thankstext[type=corresp,id=cor1]{Corresponding author.}
\address{Faculty of Informatics, \institution{University of Debrecen}, P.O. Box 12,
4010~Debrecen, \cny{Hungary}}



\markboth{I. Fazekas et al.}{Taylor's power law for the $N$-stars network evolution model}

\begin{abstract}
Taylor's power law states that the variance function decays as a power
law. It is observed for population densities of species in ecology. For
random networks another power law, that is, the power law degree
distribution is widely studied. In this paper the original Taylor's
power law is considered for random networks. A precise mathematical
proof is presented that Taylor's power law is asymptotically true for
the $N$-stars network evolution model.
\end{abstract}

\received{\sday{13} \smonth{3} \syear{2019}}
\revised{\sday{9} \smonth{8} \syear{2019}}
\accepted{\sday{9} \smonth{8} \syear{2019}}
\publishedonline{\sday{16} \smonth{9} \syear{2019}}
\end{frontmatter}

\section{Introduction}%
\label{sec1}
Taylor's power law\index{power law} is a well-known empirical pattern in ecology. Its
general form is
\begin{equation*}
V(\mu ) \approx a \mu ^{b} ,
\end{equation*}
where $\mu $ is the mean and $V(\mu )$ is the variance of a non-negative
random variable,
$a$ and $b$ are constants. $V(\mu )$ is also called the variance
function (see \cite{Morris1982}). Taylor's power law\index{power law} is called
after the British ecologist L. R. Taylor (see \cite{Taylor1961}).
Taylor's power law\index{power law} is observed for population densities of hundreds of\vadjust{\goodbreak}
species in ecology. It is 
observed in medical sciences, demography
(\cite{Cohen2018}), physics, finance (for an overview see
\cite{Kertesz2008}). Most papers on the topic present empirical studies,
but some of them offer models as well (e.g. \cite{Cohen2018} for
mortality data, \cite{Cohen2013} for population dynamics,
\cite{Kertesz2008} for complex systems). We mention that in the theory
of complex systems Taylor's power law\index{power law} is called `fluctuation scaling',
$V(\mu )$ is called the fluctuation and $\mu $ is the average. There are
papers studying Taylor's power law\index{power law} on networks (see, e.g.
\cite{barabasi2004}). In those papers Taylor's law concerns some
random variable produced by a certain process on the network.

However, there is another power law\index{power law} for networks. There are large
networks satisfying $p_{k} \sim Ck^{-\gamma }$ as $k\to \infty $, where
$p_{k}$ is the probability that a node has degree $k$. This relation is
often referred to as a power-law degree distribution or a scale-free
network. Here and in what follows $a_{k} \sim b_{k}$ means that
$\lim_{k\to \infty } a_{k}/b_{k} =1$. In their seminal paper
\cite{barabasi} Barab\'{a}si and Albert list several
scale-free large networks (actor collaboration, WWW, power grid, etc.), they introduce
the preferential attachment\index{preferential attachment} model and give an argument and numerical
evidence that the preferential attachment\index{preferential attachment} rule leads to a scale-free
network. A short description of the preferential attachment\index{preferential attachment} network
evolution model is the following. At every time step $t=2,3,\dots $ a
new vertex with $N$ edges is added to the existing graph so that the
$N$ edges link the new vertex to $N$ old vertices. The probability
$\pi _{i}$ that the new vertex will be connected to the old vertex
$i$ depends on the degree $d_{i}$ of vertex $i$,\index{vertex} so that $\pi _{i}= d
_{i}/\sum_{j} d_{j}$, where $\sum_{j} d_{j}$ is the cumulated sum of
degrees. A rigorous definition of the preferential attachment\index{preferential attachment} model was
given in \cite{bollobas}, where a mathematical proof of the power
law\index{power law} degree distribution was also presented. The idea of preferential
attachment\index{preferential attachment} and the scale-free property incited enormous research
activity. The mathematical theory is described in the monograph
\cite{hofstad} written by van der Hofstad (see also
\cite{durrett} and \cite{Chung-Lu}). The general aspects of
network theory are included in the comprehensive book
\cite{barabasiBook} by A. L. Barab\'{a}si.

There are lot of modifications of the preferential attachment\index{preferential attachment} model,
here we can list only a few of them. The following general graph
evolution model was introduced by Cooper and Frieze in
\cite{cooper}. At each time step either a new vertex or an old one
generates new edges. In both cases the terminal vertices can be chosen
either uniformly or according to the preferential attachment\index{preferential attachment} rule. In
\cite{BaMo2,FIPB2} and \cite{FIPB3} the ideas
of Cooper and Frieze \cite{cooper} were applied, but instead of
the original preferential attachment\index{preferential attachment} rule, the terminal vertices were
chosen according to the weights of certain cliques.

In several cases the connection of two edges in a network can be
interpreted as co-operation (collaboration). For example in the movie
actor network two actors are connected by an edge if they have appeared
in a film together. In the collaboration graph of scientists an edge
connects two people if they have been co-authors of a paper (see, e.g.
\cite{durrett}). In social networks, besides connections of two
members, other structures are also important. In \cite{Nastos} or
\cite{BaMo2} cliques are considered to describe co-operations. In
a clique any two vertices are connected, that is, any two members of the
clique co-operate. However, in real-life examples, in a co-operation the
members can play different roles. In a team usually one person plays
central role and the other ones play peripheral roles.
Trying to handle this
situation and to find a mathematically tractable model leads to the study
of star-like structures, see \cite{Nstar}.

In \cite{Nstar} the concept of \cite{FIPB2} was applied but
instead of cliques, star-like structures were considered. A team has
star structure if there is a head of the team and all other members are
connected to him/her. We call a graph $N$-star graph if it has $N$
vertices, one of them is called the central vertex,\index{central vertex} the remaining
$N-1$ vertices are called peripheral vertices, and it has $N-1$ edges.
The edges are directed, they start from the $N-1$ peripheral vertices
and their end point is the central vertex.\index{central vertex} In \cite{Nstar} the
following $N$-stars network evolution model was presented. In this model
at each step either a new $N$-star is constructed or an old one is
selected 
(activated) again. When $N$ vertices form an $N$ star,
then we say that they are in interaction (in other words they
co-operate). During the evolution, a vertex\index{vertex} can be in interaction
several times. We define for any vertex its central weight\index{central weight} and its
peripheral weight.\index{peripheral ! weight}\index{weight peripheral} The central weight\index{central weight} of a vertex\index{vertex} is $w_{1}$, if the
vertex\index{vertex} is a central vertex\index{central vertex} in interactions $w_{1}$ times. The
peripheral weight\index{peripheral ! weight}\index{weight peripheral} of a vertex\index{vertex} is $w_{2}$, if the vertex\index{vertex} is a peripheral
vertex\index{peripheral ! vertex}\index{vertex peripheral} in interactions $w_{2}$ times. In \cite{Nstar} asymptotic
power law distribution was proved both for $w_{1}$ and $w_{2}$.

We are interested in the following general question. Is the original
Taylor's power law\index{power law} true for random networks? First we considered data
sets of real life networks. We analysed them and the statistical
analysis showed that there are cases when Taylor's law is true and there
are cases when it is not true (our empirical results will be
published elsewhere). So we encountered the following more specific
problem: Find network structures where Taylor's power law\index{power law} is true.
To this end we analysed the above $N$-stars network evolution model.

In this paper we prove an asymptotic Taylor's power law\index{power law} for the
$N$-stars network evolution model. We shall calculate the mean and the
variance of $w_{2}$ when $w_{1}$ is fixed, and we shall see that the
variance function is asymptotically quadratic. In Section \ref{Model},
the precise mathematical description of the model and the results are
given. We recall from \cite{Nstar} the asymptotic joint
distribution of $w_{1}$ and $w_{2}$ (Proposition \ref{scalefree}). Then
we calculate the marginal distribution\index{marginal distribution} (Proposition \ref{marginal}), the
expectation\index{expectation} (Proposition~\ref{expectation}), and the second moment
(Proposition \ref{secondMoment}). The main result is Theorem~\ref{mainTHM}.
The proofs are presented in Section \ref{proofs}. Besides
mathematical proofs, we give also a numerical evidence. In Section
\ref{numerical} simulation results are presented supporting our
theoretical results.

\section{The $N$-stars network evolution model and the main results}%
\label{Model} First we give a short mathematical description
of our random graph model from \cite{Nstar}.

Let $N\geq 3$ be a fixed number. We start at time $0$ with an $N$-star
graph. Throughout the paper we call a graph $N$-star graph if it has
$N$ vertices, one of them is the central vertex,\index{central vertex} the remaining $N-1$
vertices are peripheral ones, and the graph has $N-1$ directed edges.
The edges start from the $N-1$ peripheral vertices and their end point
is the central vertex.\index{central vertex} So the central vertex\index{central vertex} has in-degree $N-1$, and
each of the $N-1$ peripheral vertices has out-degree $1$. The evolution
of our graph is governed by the weights of the $N$-stars and the
$(N-1)$-stars. In our model, the initial weight\index{weight} of the $N$-star is
$1$, and the initial weights of its $(N-1)$-star sub-graphs are also
$1$. (An $(N-1)$-star sub-graph is obtained if a peripheral vertex\index{peripheral ! vertex}\index{vertex peripheral} is
deleted from the $N$-star graph. The number of these $(N-1)$-star
sub-graphs is $N-1$.)

We first explain the model on a high level, before giving a formal
definition in the next paragraphs. The general rules of the evolution
of our graph are the following. At each time step, $N$ vertices interact,
that is, they form an $N$-star. It means that we draw all edges from the
peripheral vertices to the central vertex\index{central vertex} so that the vertices will form
an $N$-star graph. During the evolution we allow parallel edges. When
$N$ vertices interact, not only new edges are drawn, but the weights of
the stars are also increased. At the first interaction of $N$ vertices
the newly created $N$-star gets weight $1$, and its new $(N-1)$-star
sub-graphs also get weight $1$.\index{weight} If an $(N-1)$-star sub-graph is not
newly created, then its weight\index{weight} is increased by $1$. When an existing
$N$-star is selected (activated) again, then its weight\index{weight} and the weights
of its $(N-1)$-star sub-graphs are increased by $1$. So the weight\index{weight} of
an $N$-star is the number of its activations. We can see that the weight\index{weight}
of an $(N-1)$-star is equal to the sum of the weights of the $N$-stars
containing it. The weights will play crucial role in our model. The
higher the weight\index{weight} of a star the higher the chance that it will be
selected (activated) again.\looseness=1

Now we describe the details of the evolution steps of our graph. We have
two options in every step of the evolution. \texttt{Option I} has
probability $p$. In this case we add a new vertex, and it interacts with
$N-1$ old vertices. \texttt{Option II} has probability $1-p$. In this
case we do not add any new vertex, but $N$ old vertices interact. Here
$0<p\leq 1$ is fixed.

\texttt{Option I.} In this case, that is, when a new vertex is born, we
have again two possibilities: \texttt{I/1} and \texttt{I/2}.

\texttt{I/1.} The first possibility, which has probability $r$, is the
following. (Here $0\leq r\leq 1$ is fixed.) We choose one of the
existing $(N-1)$-star sub-graphs according to the preferential
attachment\index{preferential attachment} rule, and its $N-1$ vertices and the new vertex will
interact. Here the preferential attachment\index{preferential attachment} rule means that an
$  (N-1  )$-star of weight\index{weight} $v_{t}$ is chosen with probability
$v_{t}/\sum_{h} v_{h}$, where $\sum_{h} v_{h}$ is the cumulated weight
of the $(N-1)$-stars. The interaction of the new vertex and the old
$(N-1)$-star means that they establish a new $N$-star. In this newly
created $N$-star the center will be the vertex\index{vertex} which was the center in
the old $(N-1)$-star, the former $N-2$ peripheral vertices remain
peripheral and the newly born vertex\index{vertex} will be also peripheral. A new edge
is drawn from each peripheral vertex\index{peripheral ! vertex}\index{vertex peripheral} to the central one, and then the
weights are increased by $1$. More precisely, the just created $N$-star
gets weight 1, among its $(N-1)$-star sub-graphs there are $(N-2)$ new
ones, so each of them gets weight 1, finally the weight\index{weight} of the only old
$(N-1)$-star sub-graph is increased by 1.

\texttt{I/2.} The second possibility has probability $1-r$. In this
case we choose $N-1$ old vertices uniformly at random, and they will
form an $N$-star graph with the new vertex, so that the new vertex will
be the center. The edges are drawn from the peripheral vertices to the
center. As here the newly created $N$-star graph and all of its
$(N-1)$-star sub-graphs are new, so all of them get weight\index{weight} 1.

\texttt{Option II.} In this case, that is, when we do not add any new
vertex, we have two ways again: \texttt{II/1} and \texttt{II/2}.

\texttt{II/1.} The first way has probability $q$. (Here $0\leq q
\leq 1$ is fixed.) We choose one of the existing $N$-star sub-graphs by
the preferential attachment\index{preferential attachment} rule, then draw a new edge from each of its
peripheral vertices to its center vertex.\index{vertex} Then the weight\index{weight} of the
$N$-star and the weights of its $(N-1)$-star sub-graphs are increased
by $1$. Here the preferential attachment\index{preferential attachment} rule means that an $N$-star of
weight\index{weight} $v_{t}$ is chosen with probability $v_{t}/\sum_{h} v_{h}$, where
$\sum_{h} v_{h}$ is the cumulated weight of the $N$-stars.

\texttt{II/2.} The second way has probability $1-q$. In this case we
choose $N$ old vertices uniformly at random, and they establish an
$N$-star graph. Its center is chosen again uniformly at random out of
the $N$ vertices. Then, as before, new edges are drawn from the
peripheral vertices to the central one, and the weights of the $N$-star
and its $(N-1)$-star sub-graphs are increased by $1$.

\begin{rem}
\label{remI}
For every vertex we shall use its central weight\index{central weight} and its peripheral
weight.\index{peripheral ! weight}\index{weight peripheral} The central weight\index{central weight} of a vertex\index{vertex} is $w_{1}$, if the vertex\index{vertex} was a
central vertex\index{central vertex} in interactions $w_{1}$ times. The peripheral weight\index{peripheral ! weight}\index{weight peripheral} of
a vertex\index{vertex} is $w_{2}$, if the vertex\index{vertex} was a peripheral vertex\index{peripheral ! vertex}\index{vertex peripheral} in
interactions $w_{2}$ times. We can see that the central weight\index{central weight} of a
vertex\index{vertex} is equal to $w_{1}= \dfrac{d_{1}}{N-1}$ and the peripheral weight\index{peripheral ! weight}\index{weight peripheral}
of a vertex\index{vertex} is equal to $w_{2}=d_{2}$, where $d_{1}$ denotes the
in-degree of the vertex\index{vertex} and $d_{2}$ denotes its out-degree. The weights
$w_{1}$ and $w_{2}$ describe well the role of a vertex\index{vertex} in the network.
Moreover, we use $w_{1}$ and $w_{2}$ instead of degrees to obtain
symmetric formulae that allow us to translate the result from
$w_{1}$ to $w_{2}$ and vice versa without having to change the proofs.
\end{rem}
Throughout the paper $0<p\leq 1$, $0\leq r\leq 1$, $0\leq q
\leq 1$ are fixed numbers. In our formulae the following parameters are
used.
%
\begin{align}
\label{parameters} \alpha _{11} &= pr, & \alpha _{12} &=
(1-p)q,
\nonumber
\\
\alpha _{1} &= \alpha _{11}+\alpha _{12},&
\alpha _{2} &= pr \frac{N-2}{N-1}+(1-p)q,
\nonumber
\\
\beta _{1} &= \frac{(1-p)(1-q)}{p}, & \beta _{2} &=
(N-1) \biggl[(1-r)+ \frac{(1-p)(1-q)}{p} \biggr],
\nonumber
\\
\alpha &= \alpha _{1}+\alpha _{2}, & \beta &= \beta
_{1}+\beta _{2}.
\end{align}

In \cite{Nstar} it was shown that the above evolution leads to a
scale-free graph. To describe the result, let $V_{n}$ denote the number
of all vertices and let $X   ( n,w_{1},w_{2}   )$ denote the
number of vertices with central weight $w_{1}$\index{central weight} and peripheral weight
$w_{2}$\index{peripheral ! weight} after the $n$th step.
%
\begin{prop}[Theorem 2.1 of \cite{Nstar}]
\label{scalefree}
Let $0<p<1$, $0<q<1$, $0<r<1$. Then
for any fixed $w_{1}$ and $w_{2}$ with either $w_{1}=0$ and
$1\leq w_{2}$ or $1\leq w_{1}$ and $w_{2}\geq 0$ we have
%
\begin{equation}
\dfrac{X  (n,w_{1},w_{2}  )}{V_{n}}\rightarrow x_{w_{1},w
_{2}}
\end{equation}
almost surely as $n\rightarrow \infty $, where $x_{w_{1},w_{2}}$ are
fixed non-negative numbers.

Let $w_{2}$ be fixed, then as $w_{1}\rightarrow \infty $
%
\begin{equation}
\label{w2fixthm} x_{w_{1},w_{2}}\sim A(w_{2})w_{1}^{-  (1+\frac{\beta _{2}+1}{\alpha
_{1}}  )},
\end{equation}
where
%
\begin{equation}
\label{Aform} A(w_{2})=\dfrac{1-r}{\alpha _{1}}\dfrac{1}{w_{2}!}
\dfrac{\varGamma
  (w_{2}+\frac{\beta _{2}}{\alpha _{2}}  )}{\varGamma   (\frac{
\beta _{2}}{\alpha _{2}}  )}\dfrac{\varGamma   (1+\frac{
\beta +1}{\alpha _{1}}  )}{\varGamma   (1+\frac{\beta _{1}}{
\alpha _{1}}  )}.
\end{equation}

Let $w_{1}$ be fixed. Then, as $w_{2}\rightarrow \infty $,
%
\begin{equation}
\label{w1fixthm} x_{w_{1},w_{2}}\sim C(w_{1})w_{2}^{-  (1+\frac{\beta _{1}+1}{\alpha
_{2}}  )},
\end{equation}
where
%
\begin{equation}
\label{Cform} C(w_{1})=\dfrac{r}{\alpha _{2}}\dfrac{1}{w_{1}!}
\dfrac{\varGamma   (w
_{1}+\frac{\beta _{1}}{\alpha _{1}}  )}{\varGamma   (\frac{
\beta _{1}}{\alpha _{1}}  )}\dfrac{\varGamma   (1+\frac{
\beta +1}{\alpha _{2}}  )}{\varGamma   (1+\frac{\beta _{2}}{
\alpha _{2}}  )}.
\end{equation}
Here $\varGamma $ denotes the Gamma function.
\end{prop}
%
\begin{rem}
\label{remI+}
Using $w_{1}$ and $w_{2}$, we obtained symmetric formulae in the
following sense. If we interchange subscripts $1$ and $2$ of
$\alpha $ and $\beta $ (and use $r$ instead of $1-r$), then we obtain
formulae \eqref{w1fixthm}--\eqref{Cform} from formulae
\eqref{w2fixthm}--\eqref{Aform}. Therefore we do not need new proofs when
we interchange the roles of $w_{1}$ and $w_{2}$. (Of course the basic
relations \eqref{w2fixthm}--\eqref{Aform} and
\eqref{w1fixthm}--\eqref{Cform} were proved separately. To do it we
applied the properties of our model and introduced the appropriate
parametrization given in \eqref{parameters}, see \cite{Nstar}.)
\end{rem}
%
\begin{rem}
\label{remNew}
We see that $x_{w_{1},w_{2}}$ is the asymptotic joint distribution of
the central weight\index{central weight} and the peripheral weight.\index{peripheral ! weight}\index{weight peripheral} To obtain Taylor's power
law,\index{power law} we have to find the conditional expectation\index{conditional expectation} $E_{w_{1}}$ and the
conditional second moment $M_{w_{1}}$ given that $w_{1}$ is fixed. Then
the asymptotic behaviour of $E_{w_{1}}$ and $M_{w_{1}}$ will imply that
Taylor's power law\index{power law} is satisfied asymptotically. We underline that the
asymptotic relations \eqref{w2fixthm} and \eqref{w1fixthm} do not
provide enough information to find the asymptotics of $E_{w_{1}}$
and $M_{w_{1}}$. So we need deep analysis of the joint distribution
$x_{w_{1},w_{2}}$\index{joint distribution} to obtain Taylor's power law.\index{power law}
\end{rem}
Now we turn to the new results of this paper. First we consider the
marginals of the asymptotic joint distribution $x_{w_{1},w_{2}}$.\index{joint distribution} Let
%
\begin{equation}
\label{Xwdefi} x_{w_{1},\cdot } = \sum_{l=0}^{\infty }x_{w_{1},l}
\end{equation}
be the first marginal distribution.\index{marginal distribution}
%
\begin{prop}
\label{marginal}
Let $0<p<1$, $0<q<1$, $0<r<1$. Then
%
\begin{equation}
\label{X0-} x_{0,\cdot } = \dfrac{r}{\beta _{1} +1} ,
\end{equation}
and for $w_{1}>0$
%
\begin{equation}
\label{Xw1+-} x_{w_{1},\cdot } = \sum_{i=1}^{\infty }x_{w_{1}-1,i}
\dfrac{(w_{1}-1)
\alpha _{1}+\beta _{1}}{w_{1} \alpha _{1} + \beta _{1} +1 } + x_{w_{1},0} \biggl( \dfrac{\beta _{2}}{w_{1} \alpha _{1} + \beta _{1} +1 } +1
\biggr) .
\end{equation}
Moreover
%
\begin{equation}
\label{X1.-} x_{1,\cdot } = \dfrac{\beta _{1}}{\alpha _{1} + \beta _{1} +1 } \biggl(
\dfrac{r}{
\beta _{1} +1} + \dfrac{1-r}{\beta _{1}} \biggr),
\end{equation}
and
%
\begin{equation}
\label{Xw1++-} x_{w_{1},\cdot } = \dfrac{\varGamma   (w_{1}+\frac{\beta _{1}}{
\alpha _{1}}  )}{\varGamma   (\frac{\beta _{1}}{\alpha _{1}}  )}
\dfrac{
\varGamma   (1+\frac{\beta _{1}+1}{\alpha _{1}}  )}{\varGamma
  (w_{1}+1+\frac{\beta _{1}+1}{\alpha _{1}}  )} \dfrac{1-r+
\beta _{1}}{\beta _{1}(\beta _{1} +1)}
\end{equation}
for $w_{1}>1$. We have a proper distribution, that is, $\sum_{w_{1}=0}
^{\infty }x_{w_{1},\cdot } =1$.
\end{prop}

Now we turn to the conditional expectations\index{conditional expectation} of the asymptotic
distribution. Let
%
\begin{equation}
\label{Ewdefi} E_{w_{1}} = \sum_{l=0}^{\infty }x_{w_{1},l}
l /x_{w_{1},\cdot }
\end{equation}
be the expectation\index{expectation} when the central weight $w_{1}$\index{central weight} is fixed.
%
\begin{prop}
\label{expectation}
Let $0<p<1$, $0<q<1$, $0<r<1$. Then for $w_{1}>1$ we have
%
\begin{equation}
\label{A_E+-} E_{w_{1}} = \dfrac{\varGamma   (2+\frac{\beta _{1}+1-\alpha _{2}}{
\alpha _{1}}  )}{\varGamma   (2+
\frac{\beta _{1}+1}{\alpha _{1}}  )}
\dfrac{\varGamma   (w_{1}+1+\frac{
\beta _{1}+1}{\alpha _{1}}  )}{\varGamma   (w_{1}+1+\frac{\beta
_{1}+1-\alpha _{2}}{\alpha _{1}}  )} \dfrac{A_{1}}{x_{1,\cdot }} -\frac{
\beta _{2}}{\alpha _{2}} ,
\end{equation}
where
%
\begin{align}
\label{A1+-} A_{1} &= \dfrac{r}{\beta _{1}+1-\alpha _{2}} \biggl( 1+
\dfrac{\beta
_{2}}{\alpha _{2}} \biggr) \dfrac{\beta _{1}}{\alpha _{1} + \beta _{1} +1
- \alpha _{2}}
\nonumber
\\
&\quad  + (1-r)\dfrac{\beta _{2}}{\alpha _{2}} \dfrac{1}{\alpha _{1} + \beta
_{1} +1 - \alpha _{2}}.
\end{align}
Moreover
%
\begin{equation}
\label{A_E+++} E_{w_{1}} \sim \dfrac{A_{1}}{x_{1,\cdot }}
\dfrac{\varGamma   (2+\frac{
\beta _{1}+1-\alpha _{2}}{\alpha _{1}}  )}{\varGamma   (2+\frac{
\beta _{1}+1}{\alpha _{1}}  )} w_{1}^{
\frac{\alpha _{2}}{\alpha _{1}}} .
\end{equation}
That is, the magnitude of $E_{w_{1}}$ 
approaches $ w_{1}^{\frac{\alpha _{2}}{
\alpha _{1}}}$ when $w_{1}\to \infty $.
\end{prop}

Now we turn to the conditional second moments of the asymptotic
distribution. Let
%
\begin{equation}
\label{Mwdefi} M_{w_{1}} = \sum_{l=0}^{\infty }x_{w_{1},l}
l^{2} /x_{w_{1},\cdot }
\end{equation}
be the second moment when the central weight $w_{1}$\index{central weight} is fixed.
%
\begin{prop}
\label{secondMoment}
Let $0<p<1$, $0<q<1$, $0<r<1$. Assume that $\beta _{1} +1 > 2 \alpha
_{2}$. Then for $w_{1}>1$ we have
%
\begin{align}
\label{B_M+-} M_{w_{1}} &= \dfrac{\varGamma   (2+\frac{\beta _{1}+1-2\alpha _{2}}{
\alpha _{1}}  )}{\varGamma   (2+
\frac{\beta _{1}+1}{\alpha _{1}}  )}
\dfrac{\varGamma   (w_{1}+1+\frac{
\beta _{1}+1}{\alpha _{1}}  )}{\varGamma   (w_{1}+1+\frac{\beta
_{1}+1-2\alpha _{2}}{\alpha _{1}}  )} \dfrac{B_{1}}{x_{1,\cdot }}
\nonumber
\\
&\quad  - \biggl(1+2 \frac{\beta _{2}}{\alpha _{2}} \biggr) E_{w_{1}} -
\frac{
\beta _{2}}{\alpha _{2}} \biggl(1+\frac{\beta _{2}}{\alpha _{2}} \biggr) ,
\end{align}
where
\begin{align*}
B_{1} &= \dfrac{\beta _{1}}{\alpha _{1} + \beta _{1} +1 - 2 \alpha _{2}} \dfrac{r}{\beta _{1} +1-2\alpha _{2}} \biggl(1 +
\frac{\beta _{2}}{\alpha
_{2}} \biggr) \biggl(2 + \frac{\beta _{2}}{\alpha _{2}} \biggr)
\nonumber
\\
&\quad  +\frac{\beta _{2}}{\alpha _{2}} \biggl(1 + \frac{\beta _{2}}{\alpha
_{2}} \biggr)
\dfrac{1-r}{\alpha _{1} + \beta _{1} +1 - 2 \alpha _{2}} .
\end{align*}\eject\noindent
Moreover
%
\begin{equation}
\label{B_M+++} M_{w_{1}} \sim \dfrac{B_{1}}{x_{1,\cdot }}
\dfrac{\varGamma   (2+\frac{
\beta _{1}+1-2\alpha _{2}}{\alpha _{1}}  )}{\varGamma   (2+\frac{
\beta _{1}+1}{\alpha _{1}}  )} w_{1}^{2\frac{\alpha _{2}}{\alpha
_{1}}},
\end{equation}
that is, the magnitude of $M_{w_{1}}$ 
approaches $ w_{1}^{2\frac{\alpha _{2}}{
\alpha _{1}}}$ when $w_{1}\to \infty $.
\end{prop}

Now Propositions \ref{expectation}, \ref{secondMoment} imply the main
result of this paper, that is, we obtain that Taylor's law is satisfied
asymptotically.
%
\begin{thm}
\label{mainTHM}
Let $0<p<1$, $0<q<1$, $0<r<1$. Assume that $\beta _{1} +1 > 2 \alpha
_{2}$. Then
%
\begin{equation}
M_{w_{1}} \sim C E_{w_{1}}^{2}
\end{equation}
as $w_{1}\to \infty $, where $C$ is an appropriate constant. So Taylor's
law is satisfied asymptotically with exponent $2$.
\end{thm}
%
\begin{rem}
\label{rem0}
How can we observe the above Taylor's law in practice? As $x_{w_{1},w
_{2}}$ is the asymptotic joint distribution of the central weight\index{central weight} and
the peripheral weight,\index{peripheral ! weight}\index{weight peripheral} we should consider a
network large enough for asymptotic properties to show up. Fix the
central weight\index{central weight} at $w_{1}$, calculate the expectation\index{expectation} $E_{w_{1}}$ and the
second moment $M_{w_{1}}$ of the peripheral weight.\index{peripheral ! weight}\index{weight peripheral} Then we shall find
that $M_{w_{1}}$ is approximately equal to $C E_{w_{1}}^{2}$ for large
$w_{1}$. Our simulation results in Section \ref{numerical} will show a
bit more, that is, the result 
takes place for small $w_{1}$, too.
\end{rem}
%
\begin{rem}
\label{remInf}
If $\beta _{1} +1 \le 2 \alpha _{2}$, then $M_{w_{1}}$ is not finite, so
Taylor's law is not satisfied.
\end{rem}
%
\begin{rem}
\label{remII}
Now we consider the case when we interchange the roles of $w_{1}$ and
$w_{2}$. Let $w_{2}$ be fixed and let
\begin{equation*}
x_{\cdot , w_{2}} = \sum_{l=0}^{\infty }x_{l,w_{2}},
\quad E_{w_{2}} = \sum_{l=0}^{\infty }x_{l,w_{2}}
l /x_{\cdot , w_{2}}, \quad M_{w_{2}} = \sum
_{l=0}^{\infty }x_{l,w_{2}} l^{2}
/x_{\cdot , w_{2}}
\end{equation*}
be the other marginal distribution,\index{marginal distribution} conditional expectation\index{conditional expectation} and
conditional second moment. By Remark \ref{remI+}, if we interchange
subscripts $1$ and $2$ of $\alpha $ and $\beta $ (and use $r$ instead
of $1-r$), then from Proposition \ref{marginal} we obtain the
description of the behaviour of $x_{\cdot , w_{2}}$. Similarly, from
Proposition \ref{expectation} (resp. Proposition \ref{secondMoment}) we
get the appropriate results for $E_{w_{2}}$ (resp. $M_{w_{2}}$).
Finally, Theorem \ref{mainTHM} implies the following. If $\beta _{2} +1
> 2 \alpha _{1}$, then Taylor's law is satisfied asymptotically with
exponent $2$ for $E_{w_{2}}$ and $M_{w_{2}}$, too.
\end{rem}
%
\begin{rem}
\label{remIII}
For the in-degree $d_{1}$ of a vertex\index{vertex} we have $d_{1}=(N-1) w_{1}$ and
for the out-degree we have $d_{2}=w_{2}$. Let $E_{d_{1}}$ be the
conditional expectation\index{conditional expectation} of the out-degree if $d_{1}$ is fixed and let
$M_{d_{1}}$ be the conditional second moment of the out-degree if
$d_{1}$ is fixed. Then Theorem \ref{mainTHM} implies that $M_{d_{1}}
\sim {\mathrm{{const.}}} E_{d_{1}}^{2}$ as $d_{1}\to \infty $.
Similarly, Remark \ref{remII} implies that $M_{d_{2}} \sim {\mathrm{
{const.}}} E_{d_{2}}^{2}$ as $d_{2}\to \infty $. So Taylor's power law\index{power law}
is true for the in-degrees and the out-degrees.
\end{rem}

\section{Proofs and auxiliary results}%
\label{proofs} For the joint limiting distribution we have the
following result.
%
\begin{lem}[Lemma 3.2 of \cite{Nstar}]
\label{xdw}
Let $p>0$ and let $w_{1}\ge 0$,
$w_{2}\ge 0$ with $w_{1}+w_{2}\ge 1$. Then $x_{w_{1},w_{2}}$ are
positive numbers satisfying the following recurrence relation
%
\begin{align}
\label{rekurziox(w_1,w_2)} x_{1,0} &= \dfrac{1-r}{\alpha _{1} + \beta +1}, \quad
x_{0,1} = \dfrac{r}{
\alpha _{2} + \beta +1},
\nonumber
\\
x_{w_{1},w_{2}} &= \dfrac{  (\alpha _{1}   ( w_{1}-1   )
+ \beta _{1}  )x_{w_{1}-1,w_{2}} +   (\alpha _{2}   (w_{2}-1
  ) + \beta _{2}  )x_{w_{1},w_{2}-1}}{\alpha _{1} w_{1} +
\alpha _{2} w_{2} + \beta +1}
\end{align}
if $1<w_{1}+w_{2}$.
\end{lem}
Throughout the proof we shall use the following facts on the
$\varGamma $-function.
%
\begin{equation}
\label{prudformula} \sum_{i=0}^{n}
\dfrac{\varGamma (i+a)}{\varGamma (i+b)}= \dfrac{1}{a-b+1} \biggl[\dfrac{\varGamma (n+a+1)}{\varGamma (n+b)}-
\dfrac{
\varGamma (a)}{\varGamma (b-1)} \biggr],
\end{equation}
see \cite{Prud}. Stirling's formula implies that
%
\begin{equation}
\label{GSTR} \dfrac{\varGamma   (n+a  )}{\varGamma   (n+b  )} \sim n^{-(b-a)}.
\end{equation}
The above two formulae imply that
%
\begin{equation}
\label{prudformula+} \sum_{i=0}^{\infty }
\dfrac{\varGamma (i+a)}{\varGamma (i+b)}= \frac{1}{b-a-1} \dfrac{\varGamma (a)}{\varGamma (b-1)}
\end{equation}
if $b>a+1$. The following facts on $x_{w_{1},w_{2}}$ will be useful.
%
\begin{lem}[See the proof of Theorem 3.2 of \cite{Nstar}]
Let $w_{1}=0$, then
%
\begin{gather}
\label{x01} x_{0,1}=\dfrac{r}{\alpha _{2}+\beta +1}>0,
%
\\
\label{rek0l} x_{0,l}=\dfrac{1}{l\alpha _{2}+\beta +1} \bigl((l-1)\alpha
_{2}+\beta _{2} \bigr)x_{0,l-1}, \quad
l>1,
\end{gather}
and
%
\begin{equation}
\label{gammax0l} x_{0,l} = \dfrac{r}{\alpha _{2}}\dfrac{\varGamma   (1+\frac{\beta
+1}{\alpha _{2}}  )}{\varGamma   (1+\frac{\beta _{2}}{\alpha
_{2}}  )}
\dfrac{\varGamma   (l+\frac{\beta _{2}}{\alpha _{2}}  )}{
\varGamma   (l+\frac{\alpha _{2}+\beta +1}{\alpha _{2}}  )} = C(0) \dfrac{\varGamma   (l+\frac{\beta _{2}}{\alpha _{2}}  )}{
\varGamma   (l+\frac{\alpha _{2}+\beta +1}{\alpha _{2}}  )}.
\end{equation}
When $w_{2}=0$, then we have
%
\begin{gather}
\label{x10} x_{1,0}=\dfrac{1-r}{\alpha _{1}+\beta +1}>0,
%
\\
\label{rekk0} x_{k,0}=\dfrac{1}{k\alpha _{1}+\beta +1} \bigl((k-1)\alpha
_{1}+\beta _{1} \bigr)x_{k-1,0}, \quad
k>1,
\end{gather}
and
%
\begin{equation}
\label{xk0asmp} x_{k,0} = \dfrac{(1-r)\varGamma   (1+\frac{\beta +1}{\alpha _{1}}  )}{
\alpha _{1}\varGamma   (1+\frac{\beta _{1}}{\alpha _{1}}  )}
\dfrac{
\varGamma   (k+\frac{\beta _{1}}{\alpha _{1}}  )}{\varGamma
  (k+\frac{\alpha _{1}+\beta +1}{\alpha _{1}}  )} = A(0) \dfrac{
\varGamma   (k+\frac{\beta _{1}}{\alpha _{1}}  )}{\varGamma
  (k+\frac{\alpha _{1}+\beta +1}{\alpha _{1}}  )}.
\end{equation}
If $w_{1} >0$ and $l>0$, then
%
\begin{equation}
\label{xw1l} x_{w_{1},l} =\sum_{i=1}^{l}b_{w_{1}-1,i}^{(l)}x_{w_{1}-1,i}+b_{w_{1},0}
^{(l)}x_{w_{1},0},
\end{equation}
where
%
\begin{equation}
\label{bw11i} b_{w_{1}-1,i}^{(l)}=\dfrac{(w_{1}-1)\alpha _{1}+\beta _{1}}{\alpha _{2}}
\dfrac{
\varGamma   (l+\frac{\beta _{2}}{\alpha _{2}}  )}{\varGamma
  (l+1+\frac{w_{1}\alpha _{1}+\beta +1}{\alpha _{2}}  )}\dfrac{
\varGamma   (i+\frac{w_{1}\alpha _{1}+\beta +1}{\alpha _{2}}  )}{
\varGamma   (i+\frac{\beta _{2}}{\alpha _{2}}  )},
\end{equation}
for $1\leq i\leq l$, and
%
\begin{equation}
\label{bw10} b_{w_{1},0}^{(l)}= \dfrac{\varGamma   (1+\frac{w_{1}\alpha _{1}+
\beta +1}{\alpha _{2}}  )}{\varGamma   (\frac{\beta _{2}}{
\alpha _{2}}  )}
\dfrac{\varGamma   (l+\frac{\beta _{2}}{\alpha
_{2}}  )}{\varGamma   (l+1+\frac{w_{1}\alpha _{1}+\beta +1}{
\alpha _{2}}  )}.
\end{equation}
\end{lem}
Now we turn to the proofs of the new results. First we deal with
the marginal distribution.\index{marginal distribution}
\begin{proof}[Proof of Proposition \ref{marginal}]
To calculate the marginal distribution $ x_{w_{1},\cdot } = \sum_{l=0}
^{\infty }x_{w_{1},l}$\index{marginal distribution} we shall use mathematical induction. So first
consider $x_{0,\cdot }$. Because $x_{0,0}=0$, by equation
\eqref{gammax0l} we have
\begin{align*}
x_{0,\cdot } &= \sum_{l=1}^{\infty }x_{0,l}
= \dfrac{r}{\alpha _{2}}\dfrac{
\varGamma   (1+\frac{\beta +1}{\alpha _{2}}  )}{\varGamma
  (1+\frac{\beta _{2}}{\alpha _{2}}  )} \sum_{l=1}^{\infty }
\dfrac{
\varGamma   (l+\frac{\beta _{2}}{\alpha _{2}}  ) }{\varGamma
  (l+\frac{\alpha _{2}+\beta +1}{\alpha _{2}}  )}
%
\\
&= \dfrac{r}{\alpha _{2}}\dfrac{\varGamma   (1+\frac{\beta +1}{
\alpha _{2}}  )}{\varGamma   (1+\frac{\beta _{2}}{\alpha _{2}}  )} \sum_{l=0}^{\infty }
\dfrac{\varGamma   (l+1+\frac{\beta _{2}}{
\alpha _{2}}  ) }{\varGamma   (l+2 +\frac{\beta +1}{\alpha
_{2}}  )}.
\end{align*}
By \eqref{prudformula+}, the sum in the above formula is always finite,
and we have
%
\begin{equation}
\label{X0} x_{0,\cdot }= \dfrac{r}{\alpha _{2}}\dfrac{\varGamma   (1+\frac{
\beta +1}{\alpha _{2}}  )}{\varGamma   (1+\frac{\beta _{2}}{
\alpha _{2}}  )}
\dfrac{\alpha _{2}}{\beta _{1} +1} \dfrac{\varGamma
  (1+\frac{\beta _{2}}{\alpha _{2}}  )}{\varGamma   (1+\frac{
\beta +1}{\alpha _{2}}  )} = \dfrac{r}{\beta _{1} +1} .
\end{equation}

For $w_{1} >0$, by \eqref{xw1l}, we have
%
\begin{align}
\label{Xw1} x_{w_{1},\cdot } & = \sum_{l=1}^{\infty }
\sum_{i=1}^{l} b_{w_{1}-1,i}
^{(l)}x_{w_{1}-1,i} + \sum_{l=1}^{\infty }b_{w_{1},0}^{(l)}
x_{w_{1},0} + x_{w_{1},0}
\nonumber
\\
& = \sum_{i=1}^{\infty }x_{w_{1}-1,i}
\sum_{l=i}^{\infty }b_{w_{1}-1,i}
^{(l)} + x_{w_{1},0} \sum_{l=1}^{\infty }b_{w_{1},0}^{(l)}
+ x_{w_{1},0} .
\end{align}
The coefficients $b_{k,i}$ satisfy formulae \eqref{bw11i} and
\eqref{bw10}. Therefore we shall use \eqref{prudformula+} for the two
sums in the above expression. We can see that both sums are always
finite and
\begin{align*}
&\sum_{l=i}^{\infty }  b_{w_{1}-1,i}^{(l)}
\\*
&\quad =  \dfrac{(w_{1}-1)\alpha _{1}+\beta _{1}}{\alpha _{2}} \dfrac{\varGamma
  (i+\frac{w_{1}\alpha _{1}+\beta +1}{\alpha _{2}}  )}{\varGamma
  (i+\frac{\beta _{2}}{\alpha _{2}}  )} \sum
_{l=0}^{\infty }\dfrac{
\varGamma   (l+i+\frac{\beta _{2}}{\alpha _{2}}  )}{\varGamma
  (l+i+1+\frac{w_{1}\alpha _{1}+\beta +1}{\alpha _{2}}  )}
\\
&\quad = \dfrac{(w_{1}-1)\alpha _{1}+\beta _{1}}{\alpha _{2}} \dfrac{\varGamma
  (i+\frac{w_{1}\alpha _{1}+\beta +1}{\alpha _{2}}  )}{\varGamma
  (i+\frac{\beta _{2}}{\alpha _{2}}  )} \frac{\alpha _{2}}{w
_{1} \alpha _{1} + \beta _{1} +1}
\dfrac{\varGamma   (i+\frac{\beta
_{2}}{\alpha _{2}}  )}{\varGamma   (i+\frac{w_{1}\alpha _{1}+
\beta +1}{\alpha _{2}}  )}
\\
&\quad = \dfrac{(w_{1}-1)\alpha _{1}+\beta _{1}}{w_{1} \alpha _{1} + \beta
_{1} +1 } .
\end{align*}
For the other sum, a similar calculation shows that
\begin{equation*}
\sum_{l=1}^{\infty }b_{w_{1},0}^{(l)}
= \dfrac{\beta _{2}}{w_{1} \alpha
_{1} + \beta _{1} +1} .
\end{equation*}
Insert these expressions into \eqref{Xw1} to obtain
%
\begin{equation}
\label{Xw1+} x_{w_{1},\cdot } = \sum_{i=1}^{\infty }x_{w_{1}-1,i}
\dfrac{(w_{1}-1)
\alpha _{1}+\beta _{1}}{w_{1} \alpha _{1} + \beta _{1} +1 } + x_{w_{1},0} \dfrac{\beta _{2}}{w_{1} \alpha _{1} + \beta _{1} +1 } +
x_{w_{1},0} .
\end{equation}
From this point we should proceed carefully, as we should distinguish
the case of $x_{1,0}$ and the case of $x_{w_{1},0}$ for $w_{1}>1$. From
equation \eqref{X0} $x_{0,\cdot }=\dfrac{r}{\beta _{1} +1} $, from
equation \eqref{x10} $x_{1,0}=\dfrac{1-r}{\alpha _{1}+\beta +1}$ and
$x_{0,0}= 0$, so equation \eqref{Xw1+} gives that
%
\begin{align}
 x_{1,\cdot } & = \dfrac{\beta _{1}}{\alpha _{1} + \beta _{1} +1 } x_{0,
\cdot }
+ x_{1,0} \biggl( \dfrac{\beta _{2}}{\alpha _{1} + \beta _{1} +1
} +1 \biggr) \nonumber
\\
& = \dfrac{\beta _{1}}{\alpha _{1} + \beta _{1} +1 } \dfrac{r}{\beta
_{1} +1} + \dfrac{1-r}{\alpha _{1}+\beta +1}
\dfrac{\alpha _{1}+\beta +1}{
\alpha _{1} + \beta _{1} +1 }
\nonumber
\\
& = \dfrac{\beta _{1}}{\alpha _{1} + \beta _{1} +1 } \biggl( \dfrac{r}{
\beta _{1} +1} + \dfrac{1-r}{\beta _{1}}
\biggr).
\label{X1.}
\end{align}
For $w_{1}>1$ equation \eqref{rekk0} gives us $x_{w_{1},0}=\dfrac{(w
_{1}-1)\alpha _{1}+\beta _{1}}{w_{1}\alpha _{1}+\beta +1} x_{w_{1}-1,0}$,
so equation \eqref{Xw1+} implies
\begin{equation*}
x_{w_{1},\cdot } = \dfrac{(w_{1}-1)\alpha _{1}+\beta _{1}}{w_{1} \alpha
_{1} + \beta _{1} +1 } \sum_{i=0}^{\infty }x_{w_{1}-1,i}
= \dfrac{(w
_{1}-1)\alpha _{1}+\beta _{1}}{w_{1} \alpha _{1} + \beta _{1} +1 } x_{w
_{1}-1,\cdot } .
\end{equation*}
Therefore, using 
\eqref{X1.} and recursion, for $w_{1}>1$ we obtain
that
%
\begin{align}
x_{w_{1},\cdot } & = \dfrac{(w_{1}-1)\alpha _{1}+\beta _{1}}{w_{1}
\alpha _{1} + \beta _{1} +1 } x_{w_{1}-1,\cdot }
= \prod_{k=2}^{w_{1}} \dfrac{k-1+ \frac{\beta _{1}}{\alpha _{1}}}{k + \frac{\beta _{1} +1 }{
\alpha _{1}}}
x_{1,\cdot } \nonumber
\\
& = \dfrac{\varGamma   (w_{1}+\frac{\beta _{1}}{\alpha _{1}}  )}{
\varGamma   (1+\frac{\beta _{1}}{\alpha _{1}}  )} \dfrac{
\varGamma   (2+\frac{\beta _{1}+1}{\alpha _{1}}  )}{\varGamma
  (w_{1}+1+\frac{\beta _{1}+1}{\alpha _{1}}  )} \dfrac{\beta
_{1}}{\alpha _{1} + \beta _{1} +1 } \biggl(
\dfrac{r}{\beta _{1} +1} + \dfrac{1-r}{
\beta _{1}} \biggr)
\nonumber
\\
& = \dfrac{\varGamma   (w_{1}+\frac{\beta _{1}}{\alpha _{1}}  )}{
\varGamma   (\frac{\beta _{1}}{\alpha _{1}}  )} \dfrac{\varGamma
  (1+\frac{\beta _{1}+1}{\alpha _{1}}  )}{\varGamma   (w
_{1}+1+\frac{\beta _{1}+1}{\alpha _{1}}  )} \dfrac{1-r+\beta _{1}}{
\beta _{1}(\beta _{1} +1)} .
\label{Xw1++}
\end{align}

Now we check if the sum of the values of $x_{w_{1},\cdot } $ is equal
to $1$.
\begin{align*}
x_{0,\cdot } + x_{1,\cdot } + \sum_{w_{1}=2}^{\infty }x_{w_{1},\cdot
}
&= \dfrac{r}{\beta _{1} +1} + \dfrac{\beta _{1}}{\alpha _{1} + \beta
_{1} +1 } \biggl( \dfrac{r}{\beta _{1} +1} +
\dfrac{1-r}{\beta _{1}} \biggr)
\\
&\quad + \sum_{w_{1}=2}^{\infty }\dfrac{\varGamma   (w_{1}+\frac{\beta
_{1}}{\alpha _{1}}  )}{\varGamma   (\frac{\beta _{1}}{\alpha
_{1}}  )}
\dfrac{\varGamma   (1+\frac{\beta _{1}+1}{\alpha
_{1}}  )}{\varGamma   (w_{1}+1+
\frac{\beta _{1}+1}{\alpha _{1}}  )} \dfrac{1-r+\beta _{1}}{\beta
_{1}(\beta _{1} +1)} .
\end{align*}
Here
\begin{equation*}
\sum_{w_{1}=2}^{\infty }\dfrac{\varGamma   (w_{1}+\frac{\beta _{1}}{
\alpha _{1}}  )}{\varGamma   (w_{1}+1+\frac{\beta _{1}+1}{
\alpha _{1}}  )} =
\alpha _{1} \dfrac{\varGamma   (2+\frac{
\beta _{1}}{\alpha _{1}}  )}{\varGamma   (2+\frac{\beta _{1}+1}{
\alpha _{1}}  )} .
\end{equation*}
Therefore, after some calculation, we get
%
\begin{equation}
\label{sum1} \sum_{w_{1}=0}^{\infty }x_{w_{1},\cdot }
=1,
\end{equation}
so we have a proper distribution.
\end{proof}
Now we consider the expectation.\index{expectation}
\begin{proof}[Proof of Proposition \ref{expectation}]
We calculate
\begin{equation*}
E_{w_{1}} = \sum_{l=0}^{\infty }x_{w_{1},l}
l /x_{w_{1},\cdot }
\end{equation*}
that is, the expectation\index{expectation} when the central weight $w_{1}$\index{central weight} is fixed. We
shall calculate the value of
%
\begin{equation}
\label{Awdefi} A_{w_{1}} = \sum_{l=0}^{\infty }x_{w_{1},l}
\biggl(l + \frac{\beta
_{2}}{\alpha _{2}} \biggr) .
\end{equation}
As
%
\begin{equation}
\label{A_E} A_{w_{1}} = x_{w_{1},\cdot } \biggl(
E_{w_{1}} +\frac{\beta _{2}}{
\alpha _{2}} \biggr) ,
\end{equation}
using $A_{w_{1}}$, we shall find the value of $E_{w_{1}}$.

Let us start with $A_{0}$. Because $x_{0,0}=0$, and using equation
\eqref{gammax0l}, we have
\begin{align*}
A_{0} & = \sum_{l=1}^{\infty }x_{0,l}
\biggl(l + \frac{\beta _{2}}{
\alpha _{2}} \biggr) = \dfrac{r}{\alpha _{2}}
\dfrac{\varGamma   (1+\frac{
\beta +1}{\alpha _{2}}  )}{\varGamma   (1+\frac{\beta _{2}}{
\alpha _{2}}  )} \sum_{l=1}^{\infty }
\dfrac{\varGamma   (l+\frac{
\beta _{2}}{\alpha _{2}}  )   (l +
\frac{\beta _{2}}{\alpha _{2}}   )}{\varGamma   (l+\frac{\alpha
_{2}+\beta +1}{\alpha _{2}}  )}
\\
& = \dfrac{r}{\alpha _{2}}\dfrac{\varGamma   (1+\frac{\beta +1}{
\alpha _{2}}  )}{\varGamma   (1+\frac{\beta _{2}}{\alpha _{2}}  )} \sum
_{l=0}^{\infty }\dfrac{\varGamma   (l+2+\frac{\beta _{2}}{
\alpha _{2}}  ) }{\varGamma   (l+2 +\frac{\beta +1}{\alpha
_{2}}  )}.
\end{align*}
By \eqref{prudformula+}, the sum in the above formula is always finite,
moreover
%
\begin{equation}
\label{A0} A_{0} = \dfrac{r}{\beta _{1} +1-\alpha _{2}}
\dfrac{\varGamma   (2+\frac{
\beta _{2}}{\alpha _{2}}  )}{\varGamma   (1+\frac{\beta _{2}}{
\alpha _{2}}  )} = \dfrac{r}{\beta _{1} +1-\alpha _{2}} \biggl(1+\frac{
\beta _{2}}{\alpha _{2}}
\biggr).
\end{equation}

If $w_{1} >0$, then by \eqref{xw1l} we have
%
\begin{align}
\label{Aw1}  A_{w_{1}} &=
 \sum_{l=1}^{\infty }\sum
_{i=1}^{l} b_{w_{1}-1,i}^{(l)}x_{w_{1}-1,i}
\biggl(l + \frac{\beta _{2}}{\alpha _{2}} \biggr) + \sum_{l=1}^{\infty
}b_{w_{1},0}^{(l)}
x_{w_{1},0} \biggl(l + \frac{\beta _{2}}{\alpha _{2}} \biggr) + x_{w_{1},0}
\frac{\beta _{2}}{
\alpha _{2}}
\nonumber
\\
&  = \sum_{i=1}^{\infty }x_{w_{1}-1,i}
\sum_{l=i}^{\infty }b_{w_{1}-1,i}
^{(l)} \biggl(l + \frac{\beta _{2}}{\alpha _{2}} \biggr) + x_{w_{1},0}
\sum_{l=1}^{\infty }b_{w_{1},0}^{(l)}
\biggl(l + \frac{\beta _{2}}{
\alpha _{2}} \biggr) + x_{w_{1},0}
\frac{\beta _{2}}{\alpha _{2}}.
\end{align}
We know that the coefficients $b_{k,i}$ satisfy formulae
\eqref{bw11i} and \eqref{bw10}. Therefore we can apply
\eqref{prudformula+} for the first two terms in the above expression.
So we can see that both sums are always finite and
\begin{align*}
& \sum_{l=i}^{\infty }b_{w_{1}-1,i}^{(l)}
\biggl(l + \frac{\beta _{2}}{
\alpha _{2}} \biggr)
\\
&\quad  = \dfrac{(w_{1}-1)\alpha _{1}+\beta _{1}}{\alpha _{2}} \dfrac{\varGamma
  (i+\frac{w_{1}\alpha _{1}+\beta +1}{\alpha _{2}}  )}{\varGamma
  (i+\frac{\beta _{2}}{\alpha _{2}}  )} \sum
_{l=0}^{\infty }\dfrac{
\varGamma   (l+i+1+\frac{\beta _{2}}{\alpha _{2}}  )}{\varGamma
  (l+i+1+\frac{w_{1}\alpha _{1}+\beta +1}{\alpha _{2}}  )}
\\
&\quad  = \dfrac{(w_{1}-1)\alpha _{1}+\beta _{1}}{\alpha _{2}} \dfrac{\varGamma
  (i+\frac{w_{1}\alpha _{1}+\beta +1}{\alpha _{2}}  )}{\varGamma
  (i+\frac{\beta _{2}}{\alpha _{2}}  )}
\\
&\qquad \times \frac{\alpha _{2}}{w_{1} \alpha _{1} + \beta _{1} +1 -\alpha _{2}} \dfrac{\varGamma   (i+1+\frac{\beta _{2}}{\alpha _{2}}  )}{
\varGamma   (i+\frac{w_{1}\alpha _{1}+\beta +1}{\alpha _{2}}  )}
\\
&\quad  = \dfrac{(w_{1}-1)\alpha _{1}+\beta _{1}}{w_{1} \alpha _{1} + \beta
_{1} +1 - \alpha _{2}} \biggl( i+\frac{\beta _{2}}{\alpha _{2}} \biggr).
\end{align*}
Similarly
\begin{equation*}
\sum_{l=1}^{\infty }b_{w_{1},0}^{(l)}
\biggl(l + \frac{\beta _{2}}{
\alpha _{2}} \biggr) = \dfrac{\alpha _{2}}{w_{1} \alpha _{1} + \beta
_{1} +1 -\alpha _{2}}
\dfrac{\varGamma   (2+\frac{\beta _{2}}{\alpha
_{2}}  )}{\varGamma   (\frac{\beta _{2}}{\alpha _{2}}  )} .
\end{equation*}
Using these expressions, from \eqref{Aw1} we get
%
\begin{align}
 A_{w_{1}} & = \sum_{i=1}^{\infty }x_{w_{1}-1,i}
\biggl(i + \frac{\beta
_{2}}{\alpha _{2}} \biggr) \dfrac{(w_{1}-1)\alpha _{1}+\beta _{1}}{w
_{1} \alpha _{1} + \beta _{1} +1 - \alpha _{2}}  \nonumber
\\
& \quad  + x_{w_{1},0} \dfrac{\alpha _{2}}{w_{1} \alpha _{1} + \beta _{1} +1
- \alpha _{2}} \frac{\beta _{2}}{\alpha _{2}} \biggl(1 +
\frac{\beta _{2}}{
\alpha _{2}} \biggr) + x_{w_{1},0} \frac{\beta _{2}}{\alpha _{2}} .
\label{Aw1+}
\end{align}
Now we should distinguish the case of $w_{1}=1$ and the case of
$w_{1}>1$. For $w_{1}=1$ we use that from equation \eqref{x10}
$x_{1,0}=\dfrac{1-r}{\alpha _{1}+\beta +1}$ and $x_{0,0}= 0$, so equation
\eqref{Aw1+} implies that
%
\begin{align}
A_{1} &=
A_{0} \dfrac{\beta _{1}}{\alpha _{1} + \beta _{1} +1 - \alpha _{2}} + x_{1,0}
\dfrac{\alpha _{2}}{\alpha _{1} + \beta _{1} +1 - \alpha _{2}} \frac{
\beta _{2}}{\alpha _{2}} \biggl(1 + \frac{\beta _{2}}{\alpha _{2}}
\biggr) + x_{1,0} \frac{\beta _{2}}{\alpha _{2}}
\nonumber
\\
&  = \dfrac{r}{\beta _{1}+1-\alpha _{2}} \biggl( 1+ \dfrac{\beta _{2}}{
\alpha _{2}} \biggr)
\dfrac{\beta _{1}}{\alpha _{1} + \beta _{1} +1 -
\alpha _{2}}
\nonumber
\\
&\quad  + \dfrac{1-r}{\alpha _{1}+\beta +1} \dfrac{\beta _{2}}{\alpha _{2}} \dfrac{
\alpha _{1}+\beta +1}{\alpha _{1} + \beta _{1} +1 - \alpha _{2}}
\nonumber
\\
&  = \dfrac{r}{\beta _{1}+1-\alpha _{2}} \biggl( 1+ \dfrac{\beta _{2}}{
\alpha _{2}} \biggr)
\dfrac{\beta _{1}}{\alpha _{1} + \beta _{1} +1 -
\alpha _{2}} + (1-r)\dfrac{\beta _{2}}{\alpha _{2}} \dfrac{1}{\alpha _{1}
+ \beta _{1} +1 - \alpha _{2}}.
\label{A1+}
\end{align}
For $w_{1}>1$ we know that $x_{w_{1},0} = \frac{(w_{1}-1)\alpha _{1}+
\beta _{1}}{w_{1} \alpha _{1} + \beta +1 } x_{w_{1}-1,0}$, therefore
equation \eqref{Aw1+} implies that
\begin{equation*}
A_{w_{1}} = \dfrac{(w_{1}-1)\alpha _{1}+\beta _{1}}{w_{1} \alpha _{1} +
\beta _{1} +1 - \alpha _{2}} \sum_{i=0}^{\infty }x_{w_{1}-1,i}
\biggl(i + \frac{\beta _{2}}{\alpha _{2}} \biggr) .
\end{equation*}
From this equation we obtain that
%
\begin{align}
 A_{w_{1}} & = \dfrac{(w_{1}-1)\alpha _{1}+\beta _{1}}{w_{1} \alpha _{1} +
\beta _{1} +1 - \alpha _{2}} A_{w_{1}-1}
= \prod_{k=2}^{w_{1}} \dfrac{k-1+
\frac{\beta _{1}}{\alpha _{1}}}{k + \frac{\beta _{1} +1 - \alpha _{2}}{
\alpha _{1}}}
A_{1} \nonumber
\\
& = \dfrac{\varGamma   (w_{1}+\frac{\beta _{1}}{\alpha _{1}}  )}{
\varGamma   (1+\frac{\beta _{1}}{\alpha _{1}}  )} \dfrac{
\varGamma   (2+\frac{\beta _{1}+1-\alpha _{2}}{\alpha _{1}}  )}{
\varGamma   (w_{1}+1+\frac{\beta _{1}+1-\alpha _{2}}{\alpha _{1}}  )} A_{1}.
\label{Aw1++}
\end{align}
Therefore, by equation \eqref{A_E}, we have
%
\begin{align}
E_{w_{1}} & = \dfrac{A_{w_{1}}}{x_{w_{1},\cdot }} -
\frac{\beta _{2}}{
\alpha _{2}}\nonumber
\\
& = \dfrac{\varGamma   (2+\frac{\beta _{1}+1-\alpha _{2}}{\alpha
_{1}}  )}{\varGamma   (2+\frac{\beta _{1}+1}{\alpha _{1}}  )} \dfrac{\varGamma   (w_{1}+1+\frac{\beta _{1}+1}{\alpha _{1}}  )}{
\varGamma   (w_{1}+1+\frac{\beta _{1}+1-\alpha _{2}}{\alpha _{1}}  )} \dfrac{A_{1}}{x_{1,\cdot }} -
\frac{\beta _{2}}{\alpha _{2}}.
\label{A_E+}
\end{align}
Therefore, by \eqref{GSTR}, the magnitude of $E_{w_{1}}$ 
approaches
$ w_{1}^{\frac{\beta _{1}+1}{\alpha _{1}}-\frac{\beta _{1}+1-\alpha _{2}}{
\alpha _{1}}} = w_{1}^{\frac{\alpha _{2}}{\alpha _{1}}}$ when $w_{1}
\to \infty $. More precisely,
%
\begin{equation}
\label{A_E++} E_{w_{1}} \sim \dfrac{A_{1}}{x_{1,\cdot }}
\dfrac{\varGamma   (2+\frac{
\beta _{1}+1-\alpha _{2}}{\alpha _{1}}  )}{\varGamma   (2+\frac{
\beta _{1}+1}{\alpha _{1}}  )} w_{1}^{
\frac{\alpha _{2}}{\alpha _{1}}} .\qedhere
\end{equation}
\end{proof}

Now we turn to the second moment.
\begin{proof}[Proof of Proposition \ref{secondMoment}]
To find the second moment
\begin{equation*}
M_{w_{1}} = \sum_{l=0}^{\infty }x_{w_{1},l}
l^{2} /x_{w_{1},\cdot }
\end{equation*}
when the central weight $w_{1}$\index{central weight} is fixed, we shall calculate the value
of
%
\begin{equation}
\label{Bwdefi} B_{w_{1}} = \sum_{l=0}^{\infty }x_{w_{1},l}
\biggl(l + \frac{\beta
_{2}}{\alpha _{2}} \biggr) \biggl(l + 1 + \frac{\beta _{2}}{\alpha _{2}}
\biggr).
\end{equation}
We can see that
%
\begin{equation}
\label{B_M} B_{w_{1}} = x_{w_{1},\cdot } \biggl(
M_{w_{1}} + \biggl(1+2 \frac{\beta
_{2}}{\alpha _{2}} \biggr) E_{w_{1}}
+ \frac{\beta _{2}}{\alpha _{2}} \biggl(1+\frac{\beta _{2}}{\alpha _{2}} \biggr) \biggr) .
\end{equation}
Therefore, using $B_{w_{1}}$, we shall find the value of $M_{w_{1}}$.

We start with $B_{0}$. As $x_{0,0}=0$, applying equation
\eqref{gammax0l}, we obtain
\begin{equation*}
B_{0} = \sum_{l=1}^{\infty }x_{0,l}
\biggl(l + \frac{\beta _{2}}{\alpha
_{2}} \biggr) \biggl(l + 1 + \frac{\beta _{2}}{\alpha _{2}}
\biggr) = \dfrac{r}{\alpha _{2}}\dfrac{\varGamma   (1+\frac{\beta +1}{\alpha
_{2}}  )}{\varGamma   (1+\frac{\beta _{2}}{\alpha _{2}}  )} \sum
_{l=0}^{\infty }\dfrac{\varGamma   (l+3+\frac{\beta _{2}}{
\alpha _{2}}  )}{\varGamma   (l+2+\frac{\beta +1}{\alpha _{2}}  )}.
\end{equation*}
By \eqref{prudformula+}, the sum in the above formula is finite if
$\beta _{1} +1 > 2 \alpha _{2}$, and in this case
%
\begin{equation}
\label{B0} B_{0} = \dfrac{r}{\beta _{1} +1-2\alpha _{2}}
\dfrac{\varGamma   (3+\frac{
\beta _{2}}{\alpha _{2}}  )}{\varGamma   (1+\frac{\beta _{2}}{
\alpha _{2}}  )} .
\end{equation}

Now turn to $B_{w_{1}}$ when $w_{1} >0$. By \eqref{xw1l},
%
\begin{align}
B_{w_{1}} & = \sum_{l=1}^{\infty }
\sum_{i=1}^{l} b_{w_{1}-1,i}^{(l)}x
_{w_{1}-1,i} \biggl(l + \frac{\beta _{2}}{\alpha _{2}} \biggr) \biggl(l + 1 +
\frac{\beta _{2}}{\alpha _{2}} \biggr) \nonumber
\\
& \quad  + \sum_{l=1}^{\infty }b_{w_{1},0}^{(l)}
x_{w_{1},0} \biggl(l + \frac{
\beta _{2}}{\alpha _{2}} \biggr) \biggl(l + 1 +
\frac{\beta _{2}}{\alpha
_{2}} \biggr) + x_{w_{1},0} \frac{\beta _{2}}{\alpha _{2}} \biggl(1
+ \frac{
\beta _{2}}{\alpha _{2}} \biggr)
\nonumber
\\
& = \sum_{i=1}^{\infty }x_{w_{1}-1,i}
\sum_{l=i}^{\infty }b_{w_{1}-1,i}
^{(l)} \biggl(l + \frac{\beta _{2}}{\alpha _{2}} \biggr) \biggl(l + 1 +
\frac{\beta _{2}}{\alpha _{2}} \biggr)
\nonumber
\\
& \quad  + x_{w_{1},0} \sum_{l=1}^{\infty }b_{w_{1},0}^{(l)}
\biggl(l + \frac{
\beta _{2}}{\alpha _{2}} \biggr) \biggl(l + 1 + \frac{\beta _{2}}{\alpha
_{2}}
\biggr) + x_{w_{1},0} \frac{\beta _{2}}{\alpha _{2}} \biggl(1 +
\frac{
\beta _{2}}{\alpha _{2}} \biggr) .
\label{Bw1}
\end{align}
We use formulae \eqref{bw11i} and \eqref{bw10}, then apply
\eqref{prudformula+} for the first two terms in the above expression.
So we obtain that both sums are finite if $w_{1} \alpha _{1} + \beta
_{1} +1 > 2 \alpha _{2}$, and
\begin{align*}
&\sum_{l=i}^{\infty }  b_{w_{1}-1,i}^{(l)}
\biggl(l + \frac{\beta _{2}}{
\alpha _{2}} \biggr) \biggl(l + 1 + \frac{\beta _{2}}{\alpha _{2}}
\biggr)
\\
&\quad  = \dfrac{(w_{1}-1)\alpha _{1}+\beta _{1}}{\alpha _{2}} \dfrac{\varGamma
  (i+\frac{w_{1}\alpha _{1}+\beta +1}{\alpha _{2}}  )}{\varGamma
  (i+\frac{\beta _{2}}{\alpha _{2}}  )} \sum
_{l=0}^{\infty }\dfrac{
\varGamma   (l+i+2+\frac{\beta _{2}}{\alpha _{2}}  )}{\varGamma
  (l+i+1+\frac{w_{1}\alpha _{1}+\beta +1}{\alpha _{2}}  )}
\\
&\quad  = \dfrac{(w_{1}-1)\alpha _{1}+\beta _{1}}{\alpha _{2}} \dfrac{\varGamma
  (i+\frac{w_{1}\alpha _{1}+\beta +1}{\alpha _{2}}  )}{\varGamma
  (i+\frac{\beta _{2}}{\alpha _{2}}  )}
\\
& \qquad  \times \frac{\alpha _{2}}{w_{1} \alpha _{1} + \beta _{1} +1 -2 \alpha
_{2}} \dfrac{\varGamma   (i+2+\frac{\beta _{2}}{\alpha _{2}}  )}{
\varGamma   (i+\frac{w_{1}\alpha _{1}+\beta +1}{\alpha _{2}}  )}
\\
&\quad  = \dfrac{(w_{1}-1)\alpha _{1}+\beta _{1}}{w_{1} \alpha _{1} + \beta
_{1} +1 -2 \alpha _{2}} \dfrac{\varGamma   (i+2+\frac{\beta _{2}}{
\alpha _{2}}  )}{\varGamma   (i+\frac{\beta _{2}}{\alpha _{2}}  )}
\end{align*}
and
\begin{align*}
&\sum_{l=1}^{\infty }b_{w_{1},0}^{(l)}
 \biggl(l + \frac{\beta _{2}}{
\alpha _{2}} \biggr) \biggl(l + 1 + \frac{\beta _{2}}{\alpha _{2}}
\biggr)
\\
&\quad  = \dfrac{\varGamma   (1+\frac{w_{1}\alpha _{1}+\beta +1}{\alpha
_{2}}  )}{\varGamma   (\frac{\beta _{2}}{\alpha _{2}}  )} \sum_{l=1}^{\infty }
\dfrac{\varGamma   (l+2+\frac{\beta _{2}}{
\alpha _{2}}  )}{\varGamma   (l+1+\frac{w_{1}\alpha _{1}+
\beta +1}{\alpha _{2}}  )}
\\
&\quad  = \dfrac{\alpha _{2}}{w_{1} \alpha _{1} + \beta _{1} +1 -2 \alpha _{2}} \dfrac{
\varGamma   (3+\frac{\beta _{2}}{\alpha _{2}}  )}{\varGamma
  (\frac{\beta _{2}}{\alpha _{2}}  )} .
\end{align*}
Inserting these expressions into \eqref{Bw1}, we get
%
\begin{align}
B_{w_{1}} &= \sum_{i=1}^{\infty }x_{w_{1}-1,i}
\biggl(i + \frac{\beta
_{2}}{\alpha _{2}} \biggr) \biggl(i + 1 + \frac{\beta _{2}}{\alpha _{2}}
\biggr) \dfrac{(w_{1}-1)\alpha _{1}+\beta
_{1}}{w_{1} \alpha _{1} + \beta _{1} +1 -2 \alpha _{2}} \nonumber
\\
& \quad  + x_{w_{1},0} \dfrac{\alpha _{2}}{w_{1} \alpha _{1} + \beta _{1} +1 -2
\alpha _{2}} \frac{\beta _{2}}{\alpha _{2}} \biggl(1
\,{+}\, \frac{\beta _{2}}{
\alpha _{2}} \biggr) \biggl(2 \,{+}\, \frac{\beta _{2}}{\alpha _{2}} \biggr) +
x_{w_{1},0} \frac{\beta _{2}}{\alpha _{2}} \biggl(1 \,{+}\, \frac{\beta _{2}}{
\alpha _{2}}
\biggr).
\label{Bw1+}
\end{align}
When $w_{1}=1$, we use that from equation \eqref{x10} $x_{1,0}=\dfrac{1-r}{
\alpha _{1}+\beta +1}$ and $x_{0,0}= 0$, so equation \eqref{Bw1+} implies
that
%
\begin{align}
B_{1} & =B_{0} \dfrac{\beta _{1}}{\alpha _{1} + \beta _{1} +1 - 2 \alpha _{2}} + x_{1,0}
\frac{\beta _{2}}{\alpha _{2}} \biggl(1 + \frac{\beta _{2}}{
\alpha _{2}} \biggr) \biggl[1+
\dfrac{2\alpha _{2} + \beta _{2}}{\alpha
_{1} + \beta _{1} +1 - 2\alpha _{2}} \biggr]
\nonumber
\\
& = \dfrac{\beta _{1}}{\alpha _{1} + \beta _{1} +1 - 2 \alpha _{2}} \dfrac{r}{
\beta _{1} +1-2\alpha _{2}} \biggl(1 + \frac{\beta _{2}}{\alpha _{2}}
\biggr) \biggl(2 + \frac{\beta _{2}}{\alpha _{2}} \biggr)
\nonumber
\\
& \quad  + \frac{\beta _{2}}{\alpha _{2}} \biggl(1 + \frac{\beta _{2}}{\alpha _{2}} \biggr)
\dfrac{1-r}{\alpha _{1} + \beta
_{1} +1 - 2 \alpha _{2}} .
\end{align}

When $w_{1}>1$, equation \eqref{Bw1+} implies that
%
\begin{equation}
\label{Bw1+>} B_{w_{1}}= \dfrac{(w_{1}-1)\alpha _{1}+\beta _{1}}{w_{1} \alpha _{1} +
\beta _{1} +1 -2 \alpha _{2}} \sum
_{i=0}^{\infty }x_{w_{1}-1,i} \biggl(i +
\frac{\beta _{2}}{\alpha _{2}} \biggr) \biggl(i + 1 + \frac{\beta
_{2}}{\alpha _{2}} \biggr) ,
\end{equation}
where we applied that $x_{w_{1},0} = \frac{(w_{1}-1)\alpha _{1}+\beta
_{1}}{w_{1} \alpha _{1} + \beta +1 } x_{w_{1}-1,0}$. From equation
\eqref{Bw1+>} we obtain that
%
\begin{align}
B_{w_{1}} & = \dfrac{(w_{1}-1)\alpha _{1}+\beta _{1}}{w_{1} \alpha _{1} +
\beta _{1} +1 -2 \alpha _{2}} B_{w_{1}-1}
= \prod_{k=2}^{w_{1}} \dfrac{k-1+ \frac{\beta _{1}}{\alpha _{1}}}{k + \frac{\beta _{1} +1 -2
\alpha _{2}}{\alpha _{1}}}
B_{1} \nonumber
\\
& = \dfrac{\varGamma   (w_{1}+\frac{\beta _{1}}{\alpha _{1}}  )}{
\varGamma   (1+ \frac{\beta _{1}}{\alpha _{1}}  )} \dfrac{
\varGamma   (2+\frac{\beta _{1}+1-2\alpha _{2}}{\alpha _{1}}  )}{
\varGamma   (w_{1}+1+\frac{\beta _{1}+1-2\alpha _{2}}{\alpha _{1}}  )} B_{1} .
\label{Bw1++}
\end{align}
So from equation \eqref{B_M} we obtain that
%
\begin{align}
 M_{w_{1}} & = \dfrac{B_{w_{1}}}{x_{w_{1},\cdot }} - \biggl(1+2
\frac{
\beta _{2}}{\alpha _{2}} \biggr) E_{w_{1}} - \frac{\beta _{2}}{\alpha
_{2}} \biggl(1+
\frac{\beta _{2}}{\alpha _{2}} \biggr) \nonumber
\\
& = \dfrac{\varGamma   (2+\frac{\beta _{1}+1-2\alpha _{2}}{\alpha
_{1}}  )}{\varGamma   (2+\frac{\beta _{1}+1}{\alpha _{1}}  )} \dfrac{\varGamma   (w_{1}+1+\frac{\beta _{1}+1}{\alpha _{1}}  )}{
\varGamma   (w_{1}+1+\frac{\beta _{1}+1-2\alpha _{2}}{\alpha _{1}}  )} \dfrac{B_{1}}{x_{1,\cdot }}
\nonumber
\\
& \quad  - \biggl(1+2 \frac{\beta _{2}}{\alpha _{2}} \biggr) E_{w_{1}} -
\frac{
\beta _{2}}{\alpha _{2}} \biggl(1+\frac{\beta _{2}}{\alpha _{2}} \biggr) .
\label{B_M+}
\end{align}
Therefore \eqref{GSTR} implies that the magnitude of $M_{w_{1}}$ 
approaches
$ w_{1}^{\frac{\beta _{1}+1}{\alpha _{1}}-\frac{\beta _{1}+1-2\alpha _{2}}{
\alpha _{1}}} = w_{1}^{2\frac{\alpha _{2}}{\alpha _{1}}}$ as $w_{1}
\to \infty $. More precisely,
%
\begin{equation}
\label{B_M++} M_{w_{1}} \sim \dfrac{B_{1}}{x_{1,\cdot }}
\dfrac{\varGamma   (2+\frac{
\beta _{1}+1-2\alpha _{2}}{\alpha _{1}}  )}{\varGamma   (2+\frac{
\beta _{1}+1}{\alpha _{1}}  )} w_{1}^{2\frac{\alpha _{2}}{\alpha
_{1}}}.\qedhere
\end{equation}
\end{proof}
\begin{proof}[Proof of Theorem \ref{mainTHM}]
Propositions \ref{expectation} and \ref{secondMoment} imply
%
\begin{equation}
\frac{M_{w_{1}}}{E_{w_{1}}^{2}} \sim \dfrac{B_{1} x_{1,\cdot }}{A_{1}
^{2}} \dfrac{\varGamma   (2+\frac{\beta _{1}+1-2\alpha _{2}}{\alpha
_{1}}  ) \varGamma   (2+\frac{\beta _{1}+1}{\alpha _{1}}  )}{
  (
\varGamma   (2+\frac{\beta _{1}+1-\alpha _{2}}{\alpha _{1}}  )   )
^{2} }
\end{equation}
as $w_{1}\to \infty $. So Taylor's law is satisfied asymptotically.
\end{proof}

\section{Numerical results}%
\label{numerical}
Here we present some numerical evidence supporting our result. The
scheme of our computer experiment is the following. We fixed the size
$N$ of the stars, the values of the probabilities $p$, $q$ and $r$ and
generated the graph as described in Section \ref{Model} up to a fixed
step $n$. Then we calculated $E_{w_{1}}$ and $M_{w_{1}}$, that is, the
expectation\index{expectation} and the second moment of peripheral weight $w_{2}$\index{peripheral ! weight} of the
vertices when their central weight $w_{1}$\index{central weight} is fixed. We visualized the
function $E_{w_{1}} \to M_{w_{1}}$ using the logarithmic scale on both axes.
According to Theorem \ref{mainTHM} the result should approximately be
a straight line with slope 2. We also calculated $E_{w_{2}}$ and
$M_{w_{2}}$, that is, the expectation\index{expectation} and the second moment of central
weight $w_{1}$\index{central weight} of the vertices when their peripheral weight
$w_{2}$\index{peripheral ! weight} was fixed. We visualized the function $E_{w_{2}} \to M_{w_{2}}$
using the logarithmic scale on both axes. This curve also should
approximately be a straight line with slope 2.

In the following five experiments we used various parameter sets. The
step size was always $n=10^{8}$. One can check that in these five
examples the conditions $\beta _{1} +1 > 2 \alpha _{2}$ and $\beta _{2} +1
> 2 \alpha _{1}$ are satisfied. In each case we see that both
$E_{w_{1}} \to M_{w_{1}}$ and $E_{w_{2}} \to M_{w_{2}}$ are
approximately straight lines on the log-log scale.

\begin{exper}
Here $N=4$, $p=0.4$, $q=0.4$, $r=0.4$. On Figure \ref{f1} we see that
both $E_{w_{1}} \to M_{w_{1}}$ (left) and $E_{w_{2}} \to M_{w_{2}}$
(right) are approximately straight lines on the log-log scale.
%
\begin{figure}[t]
\includegraphics{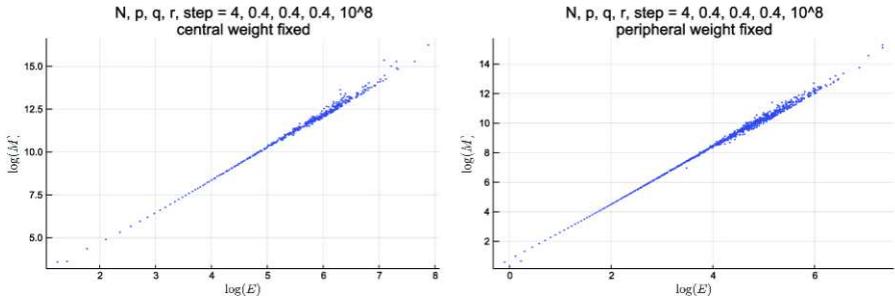}
\caption{$E_{w_{1}} \to M_{w_{1}}$ (left) and $E_{w_{2}} \to M_{w_{2}}$
(right) on the log-log scale, when $N=4$, $p=0.4$, $q=0.4$, $r=0.4$, and $n=10^{8}$\xch{}{.}}
\label{f1}
\end{figure}
\end{exper}
%
\begin{exper}
$N=5$, $p=0.4$, $q=0.4$ and $r=0.4$. The results can be seen on Figure
\ref{f3}.
%
\begin{figure}[t]
\includegraphics{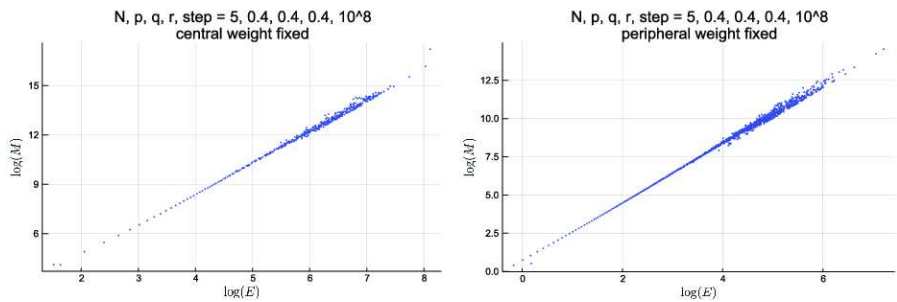}
\caption{$E_{w_{1}} \to M_{w_{1}}$ (left) and $E_{w_{2}} \to M_{w_{2}}$
(right) on the log-log scale, when $N=5$, $p=0.4$, $q=0.4$, $r=0.4$, and $n=10^{8}$\xch{}{.}}
\label{f3}
\end{figure}
\end{exper}
%
\begin{exper}
$N=5$, $p=0.5$, $q=0.5$ and $r=0.5$. The results can be seen on Figure
\ref{f4}.
%
\begin{figure}[t!]
\includegraphics{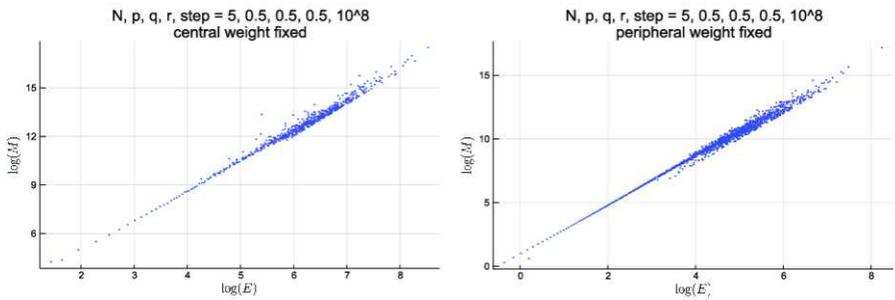}
\caption{$E_{w_{1}} \to M_{w_{1}}$ (left) and $E_{w_{2}} \to M_{w_{2}}$
(right) on the log-log scale, when $N=5$, $p=0.5$, $q=0.5$, $r=0.5$, and $n=10^{8}$\xch{}{.}}
\label{f4}
\end{figure}
\end{exper}\newpage
%
\begin{exper}
$N=6$, $p=0.3$, $q=0.6$ and $r=0.3$. The results can be seen on Figure
\ref{f5}.
%
\begin{figure}[t!]
\includegraphics{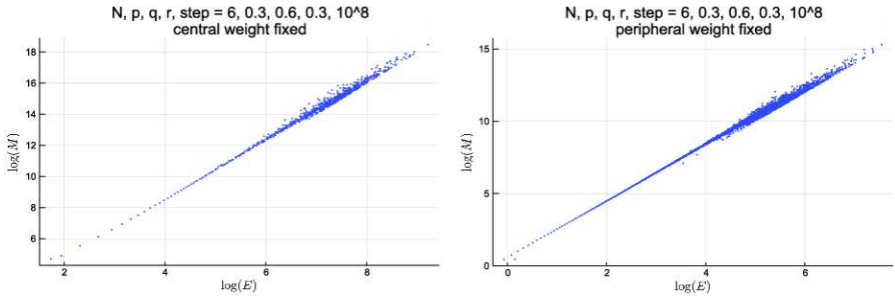}
\caption{$E_{w_{1}} \to M_{w_{1}}$ (left) and $E_{w_{2}} \to M_{w_{2}}$
(right) on the log-log scale, when $N=6$, $p=0.3$, $q=0.6$, $r=0.3$, and $n=10^{8}$\xch{}{.}}
\label{f5}\vspace*{-1pt}
\end{figure}
\end{exper}
%
\begin{exper}
$N=10$, $p=0.4$, $q=0.2$ and $r=0.7$. The results can be seen on Figure
\ref{f6}.
%
\begin{figure}[t!]
\includegraphics{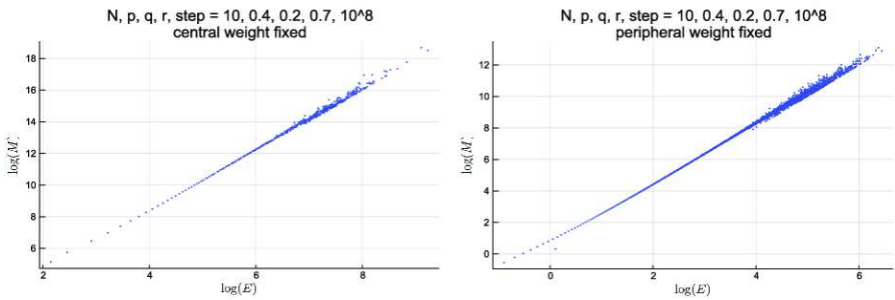}
\caption{$E_{w_{1}} \to M_{w_{1}}$ (left) and $E_{w_{2}} \to M_{w_{2}}$
(right) on the log-log scale, when $N=10$, $p=0.4$, $q=0.2$, $r=0.7$, and $n=10^{8}$\xch{}{.}}
\label{f6}\vspace*{-1pt}
\end{figure}
\end{exper}
Finally, we show a numerical result when the conditions of Theorem
\ref{mainTHM} are not satisfied.
%
\begin{exper}
Let $N=5$, $p=0.9$, $q=0.5$ and $r=0.9$, $n=10^{8}$. On Figure
\ref{f7} one can see that Taylor's power law\index{power law} is not satisfied.
%
\begin{figure}[t!]
\includegraphics{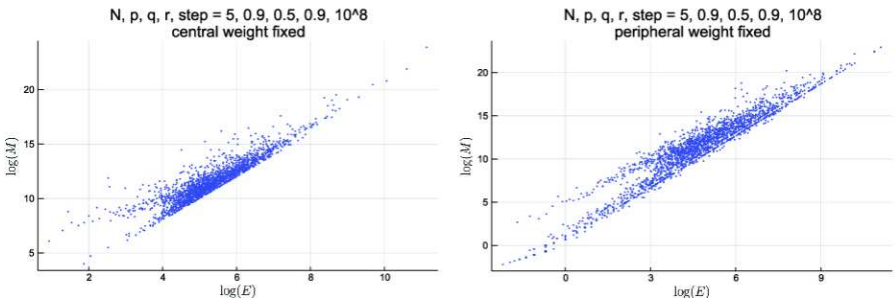}
\caption{A case when Taylor's power law is not satisfied. $E_{w_{1}} \to M_{w_{1}}$ (left) and $E_{w_{2}} \to M_{w_{2}}$
(right) on the log-log scale, when $N=5$, $p=0.9$, $q=0.5$, $r=0.9$, and $n=10^{8}$\xch{}{.}}
\label{f7}\vspace*{-1pt}
\end{figure}
\end{exper}


\begin{acknowledgement}[title={Acknowledgments}]
The authors are grateful to the referees and to the editor for the
careful reading of the paper and for the valuable suggestions.
\end{acknowledgement}

\begin{funding}
The research was supported by the construction \gnumber[refid=GS1]{EFOP-3.6.3-VEKOP-16-2017-00002};
the project was supported by the \gsponsor[id=GS1,sponsor-id=501100000780]{European Union},
co-financed by the \gsponsor[id=GS2,sponsor-id=501100004895]{European Social Fund}.
\end{funding}

\end{document}